\documentclass[a4paper,11pt,times,authoryears,reqno]{amsart}
\usepackage{geometry}
\usepackage{bm}
\usepackage{amssymb}
\usepackage{graphicx}
\usepackage{mathrsfs}
\usepackage{subcaption}
\usepackage{epstopdf}
\usepackage{enumerate}
\usepackage{booktabs}
\usepackage{adjustbox}
\usepackage[colorlinks,citecolor=blue]{hyperref}

\newtheorem{theorem}{Theorem}[section]
\newtheorem{lemma}[theorem]{Lemma}

\theoremstyle{definition}

\theoremstyle{remark}

\newtheorem{proposition}[theorem]{Proposition}

\numberwithin{equation}{section}

      \def\b{\beta}        
\def\th{\theta}

\def\no{\nonumber}

  \def\no{\nonumber}

\def\b0{{\mathbf 0}}   
\def\bW{{\mathbf W}}

\def\bA{{\mathbf A}} \def\bC{{\mathbf C}} \def\bD{{\mathbf D}}

\def\bH{{\mathbf H}} 
 
\def\bI{{\mathbf I}}
\def\bL{{\mathbf L}}  \def\bP{{\mathbf P}} \def\bQ{{\mathbf Q}}
   
 \def\bW{{\mathbf W}} \def\bX{{\mathbf X}}

  \def\bh{{\mathbf h}}

\def\bx{{\mathbf x}} \def\by{{\mathbf y}} \def\bz{{\mathbf z}}

\def\bbeta{{\boldsymbol{\beta}}}

\def\hbpi{\widehat{\boldsymbol \pi}}
\def\th{\widehat{\theta}}
\def\btheta{{\boldsymbol{\theta}}}

\def\bvt{\boldsymbol{\vartheta}}

 \def\bdelta{{\boldsymbol{\delta}}}
\def\bzeta{{\boldsymbol{\zeta}}} \def\bgamma{{\boldsymbol{\gamma}}}
\def\bLambda{{\boldsymbol{\Lambda}}}   
\def\bDelta{{\boldsymbol{\Delta}}}
 \def\bGamma{{\boldsymbol{\Gamma}}}

\def\bzeta{\boldsymbol \zeta}       
  \def\btheta{{\boldsymbol{\theta}}}

 \def\tbbeta{\widetilde{\boldsymbol \beta}}

\def\hbbeta{\widehat{\boldsymbol \beta}}

\def\bel{\begin{eqnarray}\label}  \def\eel{\end{eqnarray}}
\def\bes{\begin{eqnarray*}}  \def\ees{\end{eqnarray*}}
\def\bLd{ {\boldsymbol {L}_d} }

\def\tI{\textrm{I}}
\def\H{{\mathbb H}} 
\def\E{{\mathbb E}}

\def\hbUR{\bm{\widehat\beta}_{\textrm{UR}}}
\def\hbRE{\bm{\widehat\beta}_{\textrm{RE}}}
\def\hbPT{\bm{\widehat\beta}_{\textrm{PT}}}
\def\hbPTE{\bm{\widehat\beta}_{\textrm{PTE}}}

\def\hbS{\bm{\widehat\beta}_{\textrm{S}}}
\def\hbPS{\bm{\widehat\beta}_{\textrm{PS}}}
\def\hbMLE{\bm{\widehat\beta}_{\textrm{MLE}}}
\def\hbRMLE{\bm{\widehat\beta}_{\textrm{RMLE}}}

\begin{document}

\title{Preliminary Testing Derivatives of a Linear Unified Estimator in the Logistic Regression Model}

%    author one information

\author{Yasin Asar*}
\author{Bahad{\i}r Y\"{u}zba\c{s}{\i}}
\author{Mohammad Arashi}
\author{Jibo Wu}

\address[Yasin Asar]
{*Corresponding Author: Department of Mathematics-computer Sciences
Necmettin Erbakan University, Konya, Turkey}
\email{yasar@konya.edu.tr, yasinasar@hotmail.com}

\address[Bahad{\i}r Y\"{u}zba\c{s}{\i}]
{Department of Econometrics, Inonu University, Malatya 44280, Turkey}
\email{b.yzb@hotmail.com}

\address[Mohammad Arashi]
{Department of Statistics, School of Mathematical Sciences, Shahrood University of Technology, Shahrood, Iran}
\email{\textbf{m\_arashi\_stat}@yahoo.com}

\address[Jibo Wu]
{Key Laboratory of Group and Graph Theories and Applications, Chongqing University of Arts and Sciences, Chongqing, China}
\email{linfen52@126.com}

\subjclass[2010]{Primary 62J07; Secondary 62J02}

\date{}

\begin{abstract}
Recently, the well-known Liu estimator \cite{Liu1993} is attracted researchers' attention in regression parameter estimation for an ill-conditioned linear model. It is also argued that imposing sub-space hypothesis restriction on parameters improves estimation by shrinking toward non-sample information. \cite{chang} proposed the almost unbiased Liu estimator (AULE) in the binary logistic regression. In this article, some improved unbiased Liu type estimators, namely, restricted AULE, preliminary test AULE, Stein-type shrinkage AULE and its positive part for estimating the regression parameters in the binary logistic regression model are proposed based on the work \cite{chang}. The performances of the newly defined estimators are analyzed through some numerical results. A real data example is also provided to support the findings.

\medskip
\textbf{Keywords:} Biasing parameter; Logistic regression; MLE; Preliminary test estimator; Unbiased Liu estimator

\end{abstract}

\maketitle

%====          MAIN TEXT    ================
\section{Introduction}
Consider the following logistic regression model
\begin{equation} \label{model} 
y_{i} =\pi_{i} +{\it \epsilon}_{i} ,\; \; i=1,\ldots ,n, 
\end{equation} 
where 
\[\pi _{i} =\pi \left(x_{i} \right)=E\left[y_{i} \right]=\frac{e^{x_{i}} \bbeta } {1+e^{x_{i} } \bbeta }  ,\; \; \; y_{i} \sim {\rm Bernoulli}\left(\pi _{i} \right)\] 
and $\bbeta =\left(\beta _{0} ,\beta _{1} ,\ldots ,\beta _{p} \right)^{{'} } $ denotes the unknown (p+1)-vector of regression coefficients, $\bX=\left(x_{1} ,\ldots ,x_{n} \right)^{\top } $ is the $n\times \left(p+1\right)$ data matrix with $\bx_{i} =\left(1,x_{1i} ,\ldots ,x_{pi} \right)^{\top} $  and ${\it \epsilon}_{i} $'s are independent with zero mean and variance equal to $w_{i} =\pi _{i} \left(1-\pi _{i} \right)$ therefore the conditional distribution of the dependent variable follows a binomial distribution having probability $\pi_i$. 

\noindent The maximum likelihood (ML) method is generally used to estimate the coefficient vector $\bbeta $. The corresponding log-likelihood equation of model \eqref{model} is given by
\begin{equation} \label{LLhood} 
L=\mathop{\sum }\limits_{i=1}^{n} y_{i} {\rm log}\left(\pi _{i} \right)+\left(1-y_{i} \right){\rm log}\left(1-\pi _{i} \right) 
\end{equation} 
where $\pi _{i} $ is the $i^{th} $ element of the vector $\hbpi ,$ $i=1,2,...,n.$

\noindent ML estimator (MLE), or unrestricted MLE (UMLE) can be obtained by maximizing the log-likelihood equation given in \eqref{LLhood}. 

\noindent Since Equation \eqref{LLhood} is non-linear in $\bbeta $, one may use the iteratively re-weighted least squares algorithm (IRLS) as follows \cite{saleh2013}:
\begin{equation} \label{MLE_iteration} 
\hbbeta^{t+1} =\hbbeta^{t} +\left(\bX^{\top} \widehat{\bW} \bX\right)^{-1} \bX^{\top} \widehat{\bW}^{t} \left(\by-\hbpi^{t} \right) 
\end{equation} 
where $\hbpi^{t} $is the estimated values of $\hbpi$ using $\hbbeta_{t} $ and $\widehat{\bW}^{t} = {\rm diag}\left(\hat{\pi }_{i}^{t} \left(1-\hat{\pi }_{i}^{t} \right)\right)$ such that $\hat{\pi }_{i}^{t} $ is the ith element of $\hbpi^{t} $. After some algebra, Equation \eqref{MLE_iteration} can be written as follows:
\begin{equation} \label{1.4)} 
\hbbeta_{MLE} =\left(\bX^{\top} \widehat{\bW} \bX\right)^{-1} \bX^{\top} \widehat{\bW}\widehat{\bz} 
\end{equation} 
where $\widehat{\bz}^{\top} =\left(z_{1} \cdots z_{n} \right)$ with $\eta _{i} =x'_{i} \bbeta $ and $z_{i} =\eta _{i} +(y_{i} -\pi _{i} )(\partial \eta _{i} /\partial \pi _{i} )$.

\noindent Suppose that $\bbeta $ is subjected to lie in the sub-space restriction $\bH\bbeta =\bh$, where $\bH$ is $q\times \left(p+1\right)$ known matrix and $\bh$ is a $q\times 1$ vector of pre-specified values. Then, the corresponding restricted MLE (RMLE) has the form
\begin{equation} \label{1.5)} 
\hbRMLE=\hbMLE -\bC^{-1} \bH^{\top} \left[\bH\bC^{-1} \bH^{\top} \right]^{-1} \left(\bH\hbMLE-\bh\right) 
\end{equation} 
where $\bC=\bX^{\top}\widehat{\bW}\bX$ \cite{saleh2013}.

\noindent In reality, to select one of the UMLE or RMLE, one needs to test whether the following null hypothesis is true or not
\[H_{0} :\bH\bbeta = \bh. \] 
For testing the null hypothesis against the alternative $H_{A} :\bH\bbeta \ne \bh$, following \cite{saleh2013}, we propose the following test statistic
\[L_{n} =n\left(\bH\hbbeta_{MLE} -\bh\right)^{\top} \left[\bH\left(\frac{1}{n} \bC\right)^{-1} \bH^{\top}\right]^{-1} \left(\bH\hbbeta_{MLE} -\bh\right). \] 
As $n\to \infty $, the above test statistic has asymptotic chi-square distribution with $q$ degrees of freedom.

\noindent Now, suppose that the following regularity assumptions hold
\begin{enumerate}

\item[(A1)]: $\frac{1}{n}\bC \to \bD$, as $n\to \infty $

\item[(A2)]: $max_{1\leq i \leq n} \bx_i^{'} \bC^{-1} \bx_i = o(n)$

\item[(A3)]: $\bL_{d}^{*} \to \bL_{d} $ as $n\to \infty $ where $\bL_{d} =\bI-\left(1-d\right)^{2} \left(\bD+\bI\right)^{-2} $.
\end{enumerate}

\noindent where $\bD$ is a finite and positive definite matrix and $\bx_{i} $ is the ith row of $\bX$ and $\bL_{d}^{*} =\bI-\left(1-d\right)^{2} \left(\frac{1}{n} \bC+\bI\right)^{-2} $.

\noindent Then,~$\hbbeta_{MLE} $ is asymptotically distributed according to a normal distribution, precisely 
\begin{equation} \label{dist_MLE}
\sqrt{n} \left(\hbMLE -\beta \right)\mathop{\to }\limits^{D} N_{p+1} \left(0,\bD^{-1} \right)
\end{equation}
where ${{\stackrel{D}{\rightarrow}}}$ denotes convergence in distribution \cite{saleh2013}. And we have
\[\mathop{\lim }\limits_{n\to \infty } P\left(L_{n} \le x|H_{o} \right)=H_{q} \left(x;0\right)\] 
where $H_{q} \left(x{\kern 1pt} ;{\rm \Delta }^{2} \right)$ is the cumulative distribution function (c.d.f.) of the non-central chi-square distribution with $q$ degrees of freedom and non-centrality parameter ${\rm \Delta }^{2} /2$.

\noindent Apparently, in applications, one selects one of the extremes UMLE or RMLE depending on output of the test. Hence, the preliminary test MLE (PTMLE) is given as \cite{saleh2013}
\begin{equation} 
\hbPTE =\hbMLE -\left(\hbMLE -\hbRMLE \right)I\left(L_{n} <\chi _{q,\alpha }^{2} \right) 
\end{equation} 
where $\chi _{q,\alpha }^{2} $ is the $\alpha $ level upper value of the null distribution of the test statistic $L_{n} $ and $I\left(A\right)$ is the indicator function for the set  $A$, see \cite{saleh2006} for details.

\noindent This estimator is also known as quasi-empirical Bayes estimator \cite{saleh2013}. The preliminary test estimation approach and its shrinkage derivatives have been considered by many researches. We refer to \cite{bancroft}, \cite{judge}, \cite{akdeniz-erol}, \cite{sen}, \cite{hubert}, \cite{ks2006}, \cite{saleh2006}, \cite{ks2012}, \cite{kibria2012}, \cite{kms}, \cite{arashi2012}, \cite{arabi}, \cite{yuzbasi2016} and \cite{yuzbasi2017} to mention a few.

\noindent When the multicollinearity exists, the use of logistic ridge estimator \cite{schaefer} and logistic Liu estimator \cite{mason2012}, to propose well performed estimators as alternatives to the MLE, has been frequently seen in the studies of researchers (\cite{asar2016}, \cite{nagarajah2017}, \cite{wu2016} etc.). However, since the occurrence of $H_{o} $ is under suspicious, the use of the preliminary test counterparts is neglected. 

\noindent In this paper, we consider the estimation of $\bbeta $ in the logistic regression model by developing preliminary test estimators of a set of almost unbiased Liu type estimators. To be updated and avoiding from doing repetitive work, the building block of our problem will be based on the very recent study due to \cite{chang}. The author proposed the almost unbiased Liu estimator (AULE) which is the unrestricted estimator of $\bbeta $ considered in this study and we denote it by UR  in the logistic regression model as follows:
\begin{equation} \label{1.7)} 
\hbUR =\bL_{d} \hbMLE. 
\end{equation} 
We refer to \cite{akdeniz-erol}, \cite{mason2012}, \cite{duran}, \cite{asar2016} and \cite{wu2016} for recent studies in Liu regression. 

\noindent Following \cite{kibria2012} and \cite{roozbeh}, we define the preliminary test unbiased Liu estimator (PT) having the following form
\begin{equation} \label{1.8)} 
\hbPT =\hbUR -\left(\hbUR -\hbRE \right)I\left(L_{n} <\chi _{q,\alpha }^{2} \right) 
\end{equation} 
where
\begin{equation} \label{1.9)} 
\hbRE = \bL_{d} \hbRMLE.   
\end{equation} 
which is the restricted almost unbiased Liu estimator (RE) proposed by \cite{wu2017}.

\noindent Although the focus of this study is the preliminary test estimator, due to its discontinuous nature and strong dependency to the level of significance $\alpha $, it is of interest to consider $\alpha $-free and continuous derivatives, namely shrinkage estimators. 

\noindent The preliminary test unbiased Liu estimator (PTE) has two extreme choices, namely, the $\hbUR$ and $\hbRE$. A compromise approach can be suggested by using the Stein-type shrinkage unbiased Liu estimator (S) of $\bbeta $ as
\begin{equation} \label{S} 
\hbbeta_{S} =\hbUR -\left(\hbUR-\hbRE \right)cL_{n}^{-1}  
\end{equation} 
where $c=q-2$.

\noindent The Stein-type estimator S will provide uniform improvement over UR, however it is not a convex combination of two extremes. This estimator has the disadvantage that the shrinkage factor $\left(1-cL_{n}^{-1} \right)$ becomes negative for $L_{n} <c$. Following \cite{arashi2014}, we define the positive-rule shrinkage unbiased Liu estimator (PS) of $\bbeta $  in the logistic regression model as follows 
\begin{eqnarray} \label{PS} 
\hbPS &=&\hbRE+\left(1-cL_{n}^{-1} \right)I\left(L_{n} >c\right)\left(\hbUR -\hbRE \right) \no \\
&=& \hbbeta_{S} - \left(\hbUR -\hbRE \right)\left(1-cL_{n}^{-1} \right)I\left(L_{n} < c\right).
\end{eqnarray} 
See \cite{saleh2006} for extensive study on shrinkage estimators. 

\section{Asymptotic Performance}
In this section, we provide the asymptotic properties of the proposed estimators. According to \cite{saleh2013}, to obtain proper discrimination between the asymptotic distributions of the estimators, we consider the local alternatives of form
\begin{equation} \label{3.1)} 
K_{\left(n\right)} :\bH\bbeta =\bh+\frac{\bgamma }{\sqrt{n} }  
\end{equation} 
where $\bgamma =\left(\gamma _{1} ,\ldots ,\gamma _{q} \right)^{\top } \in {\rm {\rm R}}^{q} $ is a fixed vector.

\noindent Under the local alternatives $\left\{K_{\left(n\right)} \right\}$ we have $L_{n} \mathop{\to }\limits^{D} \chi _{q}^{2} \left({\rm \Delta }^{2} \right)$ as $n\to \infty $, where ${\rm \bDelta }^{2} =\bgamma ^\top \left(\bH \bD^{-1} \bH^{\top}\right)^{-1} \bgamma$. 

In order to compare the estimators, we use the asymptotic distributional bias ($\mathcal{B}$), asymptotic covariance ($\bGamma$) and the asymptotic risk $\left( \mathcal{R} \right)$ expressions of the proposed estimators. 

Suppose $\tbbeta$ is an estimator of $\bbeta$. The asymptotic distributional bias of an estimator $\tbbeta$ is defined as
\begin{eqnarray*}\label{ADB}
\mathcal{B}\left( \tbbeta \right) =\mathbb{E} \underset{%
n\rightarrow \infty }{\lim }\left\{\sqrt{n}\left( \tbbeta -%
\bbeta \right) \right\}.
\end{eqnarray*}
Also, the asymptotic covariance of $\tbbeta$ is given as
\begin{eqnarray} \label{ADCov} 
\bGamma \left(\tbbeta \right)=\mathbb{E} \underset{%
n\rightarrow \infty }{\lim }\left\{n\left(\tbbeta -\bbeta \right)\left(\tbbeta -\bbeta \right)^{\top } \right\} .
\end{eqnarray} 
Moreover, the asymptotic risk of $\tbbeta$ is defined as
\begin{eqnarray*}
\mathcal{R}\left(\tbbeta\right)={\rm tr}\left\{\bW \bGamma \left(\tbbeta\right)\right\} 
\end{eqnarray*} 
where $\bW$ is a non-singular matrix and we consider that $\bW=\bI$ in this study. So we have \begin{eqnarray} \label{def_risk} 
\mathcal{R}\left(\tbbeta \right)={\rm tr}\left\{\bGamma \left(\tbbeta\right)\right\} 
\end{eqnarray}
where ${\rm tr}$ is the trace of a matrix.

\noindent For our purpose, we suppose the regularity conditions (A1-A3) hold.

\noindent Under the local alternatives $\left\{K_{\left(n\right)} \right\}$, using the fact that $\sqrt{n} \left(\hbbeta_{MLE} -\beta \right)\mathop{\to }\limits^{D} N_{p+1} \left(0,D^{-1} \right)$, as $n\to \infty $, we have the following important results which reveal the asymptotic distributions to be used in the study.
\begin{proposition}\label{prop_vector_dist} 
Let $\bvt_{1} = \sqrt{n}\left(\hbUR-%
\bbeta\right)$, $\bvt_{2} =\sqrt{n}\left( \hbRE-%
\bbeta \right)$ and $\bvt_{3} =\sqrt{n}\left( \hbUR-%
\hbRE \right)$. Under the regularity assumptions (A1-A3) and the local alternatives $\left\{ K_{(n)}\right\}$, as $n\rightarrow \infty$ we have the following joint distributions:

$$\left(
\begin{array}{c}
\bvt_{1} \\
\bvt_{3}%
\end{array}%
\right) \sim\mathcal{N}\left[ \left(
\begin{array}{c}
\bzeta \\
\bLd \bdelta%
\end{array}%
\right) ,\left(
\begin{array}{cc}
\bLd \bD^{-1} \bLd^\top & \bLd \bA \bLd^\top \\
\bLd \bA \bLd^\top & \bLd \bA \bLd^\top %
\end{array}%
\right) \right]\\,$$

$$\left(
\begin{array}{c}
\bvt_{2} \\
\bvt_{3}%
\end{array}%
\right) \sim\mathcal{N}\left[ \left(
\begin{array}{c}
\bzeta-\bLd\bdelta \\
\bLd\bdelta%
\end{array}%
\right) ,\left(
\begin{array}{cc}
\bLd \left( \bD^{-1} - \bA \right) \bLd^\top & \b0_{p+1} \\
\b0_{p+1} & \bLd \bA \bLd^\top%
\end{array}%
\right) \right],$$ \\
where $\bA = \bD^{-1}\left(\bH \bD^{-1}\bH^{\top} \right)^{-1}\bH \bD^{-1}$ and $\bdelta = \bD^{-1}\left(\bH \bD^{-1}\bH^{\top} \right)^{-1}\bgamma$ and $\bzeta = -(1-d) \left(\bD+\bI \right)^{-2}\bbeta$.
\end{proposition}
\begin{proof}
See Appendix.
\end{proof}
%%%%%%%%%%%%%%%%%%%%%%%%%%%%%%%%%%%%%%%%%%%%%%%%%%%%%%
%%%%%%%%%%%%%%%%%%%%%%%%%%%%%%%%%%%%%%%%%%%%%%%%%%%%%%
\begin{theorem}
Under the assumed regularity conditions in (A1-A3) and the local alternatives $\left\{K_{(n)}\right\}$, and also using the results of the Proposition \ref{prop_vector_dist}, the expressions for asymptotic biases of the listed  estimators are obtained as:
\label{bias}
\begin{eqnarray*}
\mathcal{B}\left(\hbUR\right) &=& \bzeta\\
\mathcal{B}\left( \hbRE \right) &=& \bzeta-\bLd\bdelta \\
\mathcal{B}\left( \hbPT \right) &=& \bzeta-\bLd\bdelta \mathbb{H}_{q+2}\left( \chi_{q,\alpha}^{2};\Delta^2 \right) \\
\mathcal{B}\left( \hbS\right) &=&\bzeta-c\bLd\bdelta \mathbb{E}\left( \chi_{q+2}^{-2}\left(\Delta^2 \right)\right) \\
\mathcal{B}\left( \hbPS \right) &=&\bzeta-\bLd \bdelta\mathbb{E}\left( \chi _{q+2}^{-2}\left(\Delta^2\right) \right)-\bLd \bdelta \mathbb{H}_{q+2}\left( \chi _{q}^{2};\Delta^2 \right) \\
&&+c\bLd \bdelta \mathbb{E}\left( \chi _{q+2}^{-2}(\Delta^2)\textrm{I}\left(\chi _{q+2}^{2}(\Delta^2) < c\right) \right)
\end{eqnarray*}%
where $\mathbb{H}_{v}\left( x;\Delta^2 \right) $ is the cumulative distribution
function of the non-central chi-squared distribution with non-centrality
parameter $\Delta^2$ and $v$ degree of freedom, and for $i=1,2$  
\begin{equation} \label{3.10)} 
E\left[\chi _{q+2i}^{-2} \left({\rm \Delta }^{2} \right)\right]=E_{R} \left(\frac{1}{q+2i-2+2R} \right)=\exp \left(-\frac{{\rm \Delta }^{2} }{2} \right)\mathop{\sum }\limits_{r\ge 0} \frac{1}{{\rm \Gamma }\left(r+1\right)} \left(\frac{{\rm \Delta }^{2} }{2} \right)^{2} \frac{1}{q+2i-2+2r} \no, 
\end{equation} 
\begin{eqnarray} 
E\left[\chi _{q+2i}^{-2} \left({\rm \Delta }^{2} \right)I\left(\chi _{q+2i}^{2} \left({\rm \Delta }^{2} \right)<k\right)\right] &=& \exp \left(-\frac{{\rm \Delta }^{2} }{2} \right)\mathop{\sum }\limits_{r\ge 0} \frac{1}{{\rm \Gamma }\left(r+1\right)} \left(\frac{{\rm \Delta }^{2} }{2} \right)^{2} \frac{1}{q+2i-2+2r}\no \\
&& \times H_{q+2i-2+2r} \left(k;0\right) \no
\end{eqnarray} 
where $E_{R} \left(\cdot \right)$ stands for the expectation with respect to a Poisson variable $R$ with parameter ${\rm \Delta }^{2} /2$.
\end{theorem}
\begin{proof}
See Appendix.
\end{proof}

%%%%%%% quadratic biases   %%%%%%%%%%%%%%%%%
Now, we define the following asymptotic quadratic bias $\left(\mathcal{QB}\right)$ of an estimator $\hbbeta^*$ of the parameter vector $\bbeta$
by converting them into the quadratic form since the bias expression of all the estimators are not in the scalar form:
\begin{equation}\label{eq:asqubi}
\mathcal{QB}\left(\hbbeta^*\right)=\mathcal{B}\left(\hbbeta^*\right)^\top \mathcal{B}\left(\hbbeta^*\right).
\end{equation}
\begin{theorem}
Under the assumed regularity conditions in (A1-A3) and the local alternatives $\left\{K_{(n)}\right\}$, and also using the results of the Proposition \ref{prop_vector_dist}, the expressions for asymptotic quadratic biases of the listed  estimators are obtained as follows:
\begin{eqnarray*}
\mathcal{QB}\left(\hbUR\right) &=& (1-d)^2 \bbeta^\top \left(\bD+\bI \right)^{-4}\bbeta\\
\mathcal{QB}\left( \hbRE \right) &=& (1-d)^2 \bbeta^\top \left(\bD+\bI \right)^{-4}\bbeta-2\bzeta^\top\bLd\bdelta+\bdelta^\top \bLd^\top \bLd \bdelta \\
\mathcal{QB}\left( \hbPT \right) &=& (1-d)^2 \bbeta^\top \left(\bD+\bI \right)^{-4}\bbeta-2\bzeta^\top\bLd\bdelta\mathbb{H}_{q+2}\left( \chi_{q,\alpha}^{2};\Delta^2 \right)\\
&&+\bdelta^\top \bLd^\top \bLd \bdelta \left[ \mathbb{H}_{q+2}\left( \chi_{q,\alpha}^{2};\Delta^2 \right)\right]^2 \\
\mathcal{QB}\left( \hbS\right) &=&(1-d)^2 \bbeta^\top \left(\bD+\bI \right)^{-4}\bbeta-2\bzeta^\top\bLd\bdelta\mathbb{E}\left( \chi _{q+2}^{-2}\left(\Delta^2\right) \right) \\
&&+\bdelta^\top \bLd^\top \bLd \bdelta \left[ \mathbb{E}\left( \chi_{q+2}^{-2}\left(\Delta^2 \right)\right)\right]^2 \\
\mathcal{QB}\left( \hbPS \right) &=&(1-d)^2 \bbeta^\top \left(\bD+\bI \right)^{-4}\bbeta+2\bLd \bdelta f(\Delta)+\bLd \bdelta f(\Delta)^2 
\end{eqnarray*}%
where $f(\Delta)=c\mathbb{E}\left\{\chi _{q+2}^{-2}(\Delta^2) \textrm{I}\left(\chi_{q+2}^{2} (\Delta^2) < c \right) \right\}-\mathbb{E}\left\{\chi_{q+2}^{-2}\left(\Delta^2 \right)\right\}-\mathbb{H}_{q+2}\left( \chi _{q}^{2};\Delta^2 \right).$
\end{theorem}
We skip the proof of this theorem since it is immediate from the previous Theorem 2.2.

%%%%%%%%% risks   %%%%%%%%%%%%%%%%%%%%%%%%%%%%%%%%%%%%%
\begin{theorem}\label{risks}
Under the assumed regularity conditions in (A1-A3), the Proposition \ref{prop_vector_dist}, and the local alternatives $\left\{K_n\right\}$, the asymptotic risks for the estimators are computed as follows:
\begin{eqnarray*}
\mathcal{R}\left( \hbUR\right)&=&
\sum_{j=1}^{p+1} \left\{\frac{\left(\lambda _{j} +d\right)^{2} \left(\lambda _{j} +d-2\right)^{2} }{\lambda _{j} \left(\lambda _{j} +1\right)^{4} }+ \frac{\left(1-d\right)^{4} \theta _{i}^{2} }{\left(\lambda _{i} +1\right)^{4} }\right\}\\
\mathcal{R}\left( \hbRE \right)&=&\sum_{j=1}^{p+1} \left\{\frac{\left(\lambda_j+2-d\right)^2\left(\lambda_j+d\right)^2}{\left(\lambda_j+1\right)^4} \left[\frac{1-\lambda_j a_{jj}}{\lambda_j}+\delta_j^2 \right] + \frac{\left(1-d\right)^2}{\left(\lambda_j+1\right)^4}\left[ \beta_j^2 + 2\beta_j^2\delta_j \left(\lambda_j+2-d\right) \left(\lambda_j+d\right)\right] \right\} \\ 
\mathcal{R}\left( \hbPT \right)&=&\mathcal{R}\left( \hbUR \right)-2{\rm tr}\left[\bL_d \bA \bL_d^\top\right]\H_{q+2}\left( \chi_{q,\alpha}^{2};\Delta^2 \right)-2{\rm tr}\left[\bL_d\bdelta\bzeta^{\top}\right]\H_{q+2}\left( \chi_{q,\alpha}^{2};\Delta^2 \right)\no\\
&&+{\rm tr}\left[\bL_d\bdelta\bdelta^{\top}\bL_d^\top\right] \left[2\H_{q+2}\left( \chi_{q,\alpha}^{2};\Delta^2 \right)-\H_{q+4}\left( \chi_{q+4}^{2};\Delta^2 \right)\right]\\
\mathcal{R}\left( \hbS\right)&=&
\mathcal{R}\left( \hbUR \right)+ {\rm tr}\left[\bLd \bA \bLd^\top\right] \left( c^2\E\left\{ \chi_{q+2}^{-4}\left(\Delta^2\right) \right\}-2c\E\left\{\chi_{q+2}^{-2}\left(\Delta^2\right)\right\}\right) \no\\
&&+{\rm tr}\left[\bLd \bdelta \bdelta^\top \bLd^\top \right] \left(c^2\E\left\{ \chi_{q+4}^{-4}\left(\Delta^2\right) \right\}+2c\E\left\{\chi_{q+2}^{2}\left(\Delta^2\right) \right\}-2c\E\left\{\chi_{q+4}^{-2}\left(\Delta^2\right) \right\}  \right) \no \\
&&-2c{\rm tr}\left[ \bLd \bdelta \bzeta^\top\right]\E\left\{ \chi_{q+2}^{2}\left(\Delta^2\right) \right\}\\
\mathcal{R}\left( \hbPS\right)&=&
\mathcal{R}\left( \hbS\right) \no\\
  &&+ {\rm tr}\left[ \bLd \bA \bLd^\top \right]\left[ \H_{q+2}\left( c;\Delta^2 \right)-c^2 \E\left\{ \left(\chi_{q+2}^{-4}(\Delta) \right)\tI \left(\chi_{q+2}^{-2}(\Delta)  < c \right)\right\} \no \right.\\
  &&\left.-2\E\left\{ \left(1-c\chi_{q+2}^{-2}(\Delta^2) \right)\tI \left(\chi_{q+2}^{-2}(\Delta)  < c \right)\right\}\right] \no\\
  &&+ {\rm tr}\left[ \bLd \bdelta \bdelta^\top \bLd^\top \right] \left[\H_{q+2}\left( c;\Delta^2 \right)-c^2\E\left\{ \left(\chi_{q+4}^{-4}(\Delta^2) \right)\tI \left(\chi_{q+2}^{-2}(\Delta^2)  < c \right)\right\}  \no \right.\\
  &&\left. +2c\E\left\{ \left(1-c\chi_{q+2}^{-2}(\Delta^2) \right)\tI \left(\chi_{q+2}^{-2}(\Delta^2)  < c \right)\right\} -2\E\left\{ \left(1-c\chi_{q+4}^{-2}(\Delta^2) \right)\tI \left(\chi_{q+4}^{-2}(\Delta^2)  < c \right)\right\}\right] \no\\
 && -2{\rm tr}\left[ \bLd \bdelta \bzeta^\top \right]\E\left\{ \left(1-c\chi_{q+4}^{-2}(\Delta^2) \right)\tI \left(\chi_{q+4}^{-2}(\Delta^2)  < c \right)\right\}
\end{eqnarray*}
%%%%%%%%%%%%%%%%%%%%%%%%%%%%%%
where
\begin{eqnarray}
E\left[\chi _{q+2i}^{-4} \left({\rm \Delta }^{2} \right)\right]&=&E_{R} \left(\frac{1}{\left(q+2i-2+2R\right)\left(q+2i-4+2R\right)} \right)\nonumber \\ &=&\exp \left(-\frac{{\rm \Delta }^{2} }{2} \right)\mathop{\sum }\limits_{r\ge 0} \frac{1}{{\rm \Gamma }\left(r+1\right)} \left(\frac{{\rm \Delta }^{2} }{2} \right)^{2} \frac{1}{\left(q+2i-2+2r\right)\left(q+2i-4+2r\right)} \no, 
\end{eqnarray} 
and  
\begin{eqnarray} 
E\left[\chi _{q+2i}^{-4} \left({\rm \Delta }^{2} \right)I\left(\chi _{q+2i}^{2} \left({\rm \Delta }^{2} \right)<k\right)\right] &=& \exp \left(-\frac{{\rm \Delta }^{2} }{2} \right)\mathop{\sum }\limits_{r\ge 0} \frac{1}{{\rm \Gamma }\left(r+1\right)} \left(\frac{{\rm \Delta }^{2} }{2} \right)^{2} \no \\ 
&& \times \frac{1}{\left(q+2i-2+2r\right)\left(q+2i-4+2r\right)}  H_{q+2i-4+2r} \left(k;0\right) \no. 
\end{eqnarray} 
for $i=1,2$.
\end{theorem}
%%%%%%%%%%%%%%%%%%%%%%%%%%%%%%%%%%%%%%%%%%%%%%%%

\noindent In the forthcoming section, to obtain better perception about the asymptotic performance of the estimators, we compare the MSE performance of the proposed estimators graphically. Indeed, we want to see how estimators are compared using the results of proposed theorems.

\section{Monte Carlo Simulation Study}

Our simulation is based on a logistic regression model with sample size n = 250. A binary response is generated from the Bernoulli distribution $Be(\bP)$ such that
\begin{equation*}
\bP_i =\frac{exp\left(\bx_i'\bbeta\right)}{1+exp\left(\bx_i'\bbeta\right)}, i=1,\dots,n,
\end{equation*}%
where $\bP_i =P(Y=1|\bx_i)$ and the predictor values $\bx_i$ are drown from the standard normal distribution with the correlation between the $j$th and $k$th components of $\bX$ equals to $0.5^{|j-k|}$.  $\bbeta=(\bbeta'_1,\bbeta'_2)'$ with $\bbeta_1=(1.5,2.5)'$ and $\bbeta_2=(\mathbf{0}_{q})'$. We consider the candidate submodel $H_0 : \bH\bbeta =\mathbf{0}$ where the first $2$ columns of $\bH$ are zeros and the $q \times q$ submatrix of $\bH$ is the identity. We define a distance between the simulation model and the candidate subspace model by $\Delta^*=||\bbeta-\bbeta_o||^2$ where $\bbeta_o=(\bbeta'_1,\mathbf{0}'_{q})'$ is the true parameter in the simulation model and $||\cdot||$ is the Euclidean norm.

\begin{figure}[!htbp]
\centering
   \includegraphics[height=24cm,width=14cm]{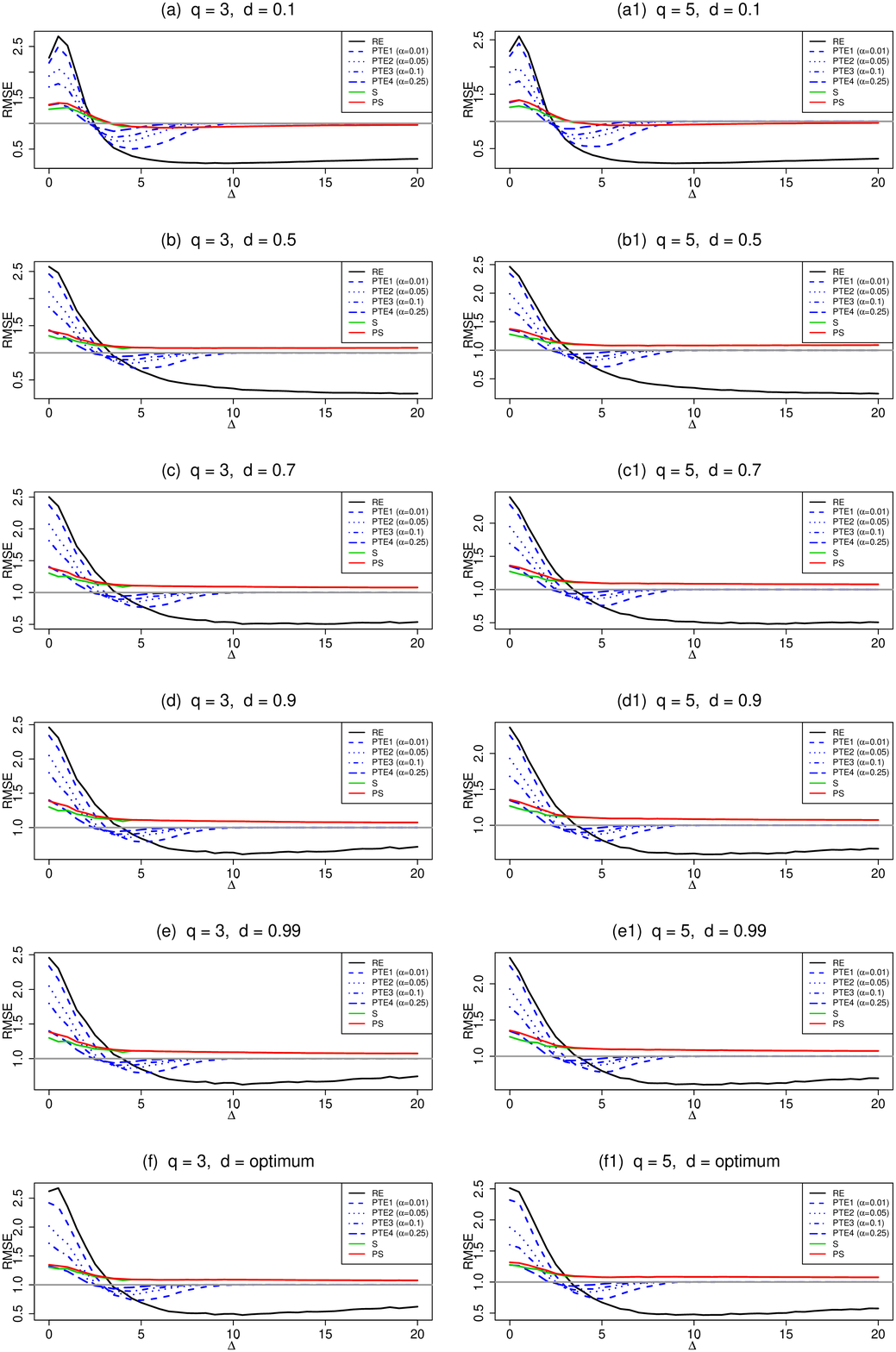}
   \caption{RMSE}
   \label{fig:rmse} 
\end{figure}

$\Delta$ changes between $0$ and $20$. $q$ is taken to be $3$ and $5$. The biasing parameter $d$ is considered to be $0.1, 0.5, 0.7, 0.9$ and $0.99$. Following \cite{alheety} and \cite{wu2017}, we also use the following estimator of $d$:
\begin{eqnarray}
d_{optimum} = 1 - \sqrt{\left(\frac{\sum_{j=1}^p\frac{1}{\lambda_j(\lambda_j+1)}}{\sum_{j=1}^p\frac{1+\lambda_j \th_j^2}{\lambda_j(\lambda_j+1)^4}}\right)}
\end{eqnarray}
where $\th_j$ is the jth element of $\hat{\btheta} ={\rm \bQ}^\top \hbMLE$ and $\bQ$ is the orthogonal matrix whose columns consist of eigenvectors of $\bD$.
The simulations are repeated 2000 times. For each replication, we compute the simulated MSE values of each estimator $\bbeta^*$ and we obtain the average MSE value using  
$$MSE\left(\bbeta^*\right)=\sum_{r=1}^{2000}\left(\bbeta^*-\bbeta \right)_r^\top \left(\bbeta^*-\bbeta \right)_r.$$
We consider the simulated relative mean squared error (RMSE) of each estimator $\bbeta^*$ as follows:
$$ RMSE\left(\bbeta^*\right) = \frac{MSE\left(\hbUR\right)}{MSE\left(\bbeta^*\right)}$$
All the computations are performed using R Statistical Package Program \cite{R2010}. The results of the simulation is summarized in Figure ~\ref{fig:rmse}.  The findings of Figure ~\ref{fig:rmse} can be summarized as follows:
\begin{enumerate} \setlength{\itemsep}{0pt} \setlength{\parskip}{0pt}

   \item [a)] When the null hypothesis is true, i.e., $\Delta^{\ast}=0$, the performance of the RE is the best. On the other hand, the RMSE of the RE slightly decreases and approaches to zero while the null hypothesis is violated. However, it mostly depends on the biasing parameter $d$. For example, if $d=0.1$, then the RMSE of the RE interestingly is not the best when $\Delta^{\ast}=0$. But, it peaks a small amount violation of the null hypothesis, after that it decreases and approaches to zero. if $d$ is relative large, say $0.9$, then the RMSE of the RE may increase while $\Delta^{\ast}$ is larger than 10. But, it still remains below to the line one when $\Delta^{\ast}$ is enough large.

   \item[b)] We investigate the performance of PTE in four aspects: $\alpha = 0.01, 0.05, 0.10, 0.25$.
In summary, if $\alpha$ is smaller, then the PTE performs better when the $\Delta^{\ast}$ is between $0$ and around $3$. For intermediate values of $\Delta^{\ast}$, the RMSE of PTE becomes worser than the UR. Finally, the RMSE of the PTE goes to one when $\Delta^{\ast}$ is large.

   \item[c)] It can be shown that the performance of PS outshines the shrinkage estimation for all values of $\Delta^{\ast}$. This is also consistent with our theory. Moreover, the RMSE of PS is only better than PTE with $\alpha=0.25$ when we assume the null hypothesis is true. On the other hand, the RMSE of the PS decreases gradually and approaches to one when the alternative hypothesis is true.

      \item[d)] We also examine that the performance of pretest and shrinkage estimations perform better when the number of "nuisance parameter" is large.

\end{enumerate}

\section{Real Data Application}

We consider the heart disease data which was also analyzed by \cite{hossain}
 and \cite{hastie}. There are 462 individuals in this dataset.  The dependent variable is an indicator variable showing that whether the individual has a coronary heart disease (chd) or not. The affecting factors are  systolic blood pressure (sbp), cumulative tobacco in kg (tobocco), low density lipoprotein cholesterol (ldl), adiposity,  family history of heart disease, a factor with levels Absent or Present (famhist),  type-A behavior (typea),  obesity,  current alcohol consumption  (alcohol)  and age.
 
 \begin{table}[ht]
\caption{Cross-validated MSE values of estimators}
\centering
\label{app}
\begin{tabular}{rrrrrrr}
  \hline
 & 0.1 & 0.5 & 0.7 & 0.9 & 0.99 & $d_{optimum}$ \\ 
  \hline
UR & 120.1662 & 132.2195 & 136.0768 & 138.0744 & 138.3248 & 135.6160 \\ 
  RE & 114.3332 & 127.1294 & 131.2303 & 133.3549 & 133.6213 & 130.7096 \\ 
  PTE1 & 114.3332 & 127.1294 & 131.2303 & 133.3549 & 133.6213 & 130.7096 \\ 
  PTE2 & 114.3498 & 127.1451 & 131.2457 & 133.3702 & 133.6365 & 130.7252 \\ 
  PTE3 & 114.4550 & 127.2421 & 131.3400 & 133.4631 & 133.7292 & 130.8204 \\ 
  PTE4 & 115.4537 & 128.1537 & 132.2234 & 134.3318 & 134.5961 & 131.7128 \\ 
  S & 116.7602 & 129.2675 & 133.2735 & 135.3485 & 135.6087 & 132.7791 \\ 
  PS & 116.7241 & 129.2121 & 133.2118 & 135.2836 & 135.5434 & 132.7184 \\ 
   \hline
\end{tabular}
\end{table}
 
Since there is no prior information, we follow \cite{hossain} and consider that the  candidate subspace is $$\bbeta_2 = \left(\text{sbp, adiposity, obesity, alcohol} \right)=(0,0,0,0).$$
We use 10-fold cross-validation  and  compute the  MSE  of each estimator and repeat this procedure 500 times. Finally, we compute the average MSE values and report them in Table ~\ref{app} for different values of the parameter $d$. According to Table  ~\ref{app}, RE has the best performance which it has the smallest MSE value and PTEs follow it. Shrinkage and its positive perform well compare to the full model estimators. All the shrinkage and preliminary test estimators has lower MSE values than UR which is also satisfied with the results of simulation and theory.

\section{Conclusion}
This paper introduced the preliminary test almost unbiased Liu, Stein-type shrinkage almost unbiased Liu and positive-rule Stein-type shrinkage almost unbiased Liu estimators in Logistic Regression model, to provide improvement upon the recent approach of \cite{chang}. We implemented a sub-space restriction on the parameter $\bbeta$ to propose improved estimation strategies. Asymptotic distributional bias and quadratic risk of the estimators are exactly given and numerical comparisons provided. We further considered the application of proposed estimators in a real data example. Numerical results confirm that the proposed positive-rule Stein-type shrinkage almost unbiased Liu estimator, which is a derivative of the preliminary one, performs significantly better than all others. Similar conclusions to \cite{saleh2013} and \cite{arashi2014} are obtained. The availability of these results in this paper should stimulate research and applications.

\clearpage
\section*{Appendix}
%%%%%%%%%%%%%%%%%%%%%%%%%%%%%%%%
\begin{lemma} \label{lem_JB}
Let $\bX$ be $q-$dimensional normal vector distributed as $%
N\left( \boldsymbol{\mu }_{x},\boldsymbol{\Sigma }%
_{q}\right) ,$ then, for a measurable function of of $\varphi ,$ we have
\begin{align*}
\mathbb{E}\left[ \bX\varphi \left( \bX^{\top}\bX%
\right) \right] =&\boldsymbol{\mu }_{x}\mathbb{E}\left[ \varphi \chi _{q+2}^{2}\left(
\Delta \right) \right] \\
\mathbb{E}\left[ \boldsymbol{XX}^{\top}\varphi \left( \bX^{\top}%
\bX\right) \right] =&\boldsymbol{\Sigma }_{q}\mathbb{E}\left[ \varphi
\chi _{q+2}^{2}\left( \Delta \right) \right] +\boldsymbol{\mu }_{x}\boldsymbol{\mu }_{x}^{\top}\mathbb{E}\left[ \varphi \chi
_{q+4}^{2}\left( \Delta \right) \right]
\end{align*}
where $\chi_{v}^{2}\left( \Delta \right)$ is a non-central chi-square distribution with $v$ degrees of freedom and non-centrality parameter $\Delta$.
\end{lemma}
\begin{proof} It can be found in \cite{judge}.\end{proof}

%%%%%%%%%%%%%%%%%%%%%%%%%%%%%%
\begin{proof}[Proof of Proposition \ref{prop_vector_dist}]
It is easy to obtain the asymptotic bias of UR as follows:
\begin{eqnarray*}
\mathcal{B}\left(\hbUR \right)&=&\mathbb{E}\left\{\underset{%
n\rightarrow \infty }{\lim }\sqrt{n}\left( \hbUR-\bbeta \right)\right\} \\
&=& \mathbb{E}\left\{\underset{%
n\rightarrow \infty }{\lim }\sqrt{n}\left( \bLd \hbMLE -\bbeta\right)\right\} \\
&=& \mathbb{E}\left\{\underset{%
n\rightarrow \infty }{\lim }\sqrt{n}\left[\bI- (1-d)^2 \left(\bD+\bI \right)^{-2} \right]\hbMLE-\bbeta\right\} \\
&=& \mathbb{E}\left\{\underset{%
n\rightarrow \infty }{\lim }\sqrt{n}\left[\left(\hbMLE-\bbeta \right) -(1-d)^2 \left(\bD+\bI \right)^{-2} \right]\hbMLE\right\} \\
&=& -(1-d)^2 \left(\bD+\bI \right)^{-2} \bbeta \\
&=& \bzeta
\end{eqnarray*}
and also using the definition of asymptotic covariance and making use of Equation (2.12) given in \cite{saleh2013}, we get the following
\begin{eqnarray} \label{cov_UR}
\boldsymbol{\Gamma} \left( \hbUR \right) &=&\mathbb{E}\left\{
\underset{n\rightarrow \infty }{\lim }{n}\left( \hbUR -\bbeta \right)\left( \hbUR -\bbeta \right) ^{\top}\right\} \no \\
&=&\mathbb{E}\left( \bvt_{1}\bvt_{1}^{\top}\right) \no \\
&=& Cov\left( \bvt_{1}\bvt_{1}^{\top}\right)+\mathbb{E}\left( \bvt_{1}\right) \mathbb{E}\left( \bvt_{1}^{\top}\right) \no \\
&=&\bLd \bD^{-1} \bLd^\top + \bzeta\bzeta^\top
\end{eqnarray}

Thus, $\bvt_{1}  \sim\mathcal{N} \left(\bzeta, \bLd \bD^{-1} \bLd^\top \right).$ \\
In \cite{saleh2013}, it is given that the asymptotic distribution of RMLE is $\hbRMLE \sim\mathcal{N} \left(-\bdelta, \bD^{-1}-\bA \right)$ where $\bA = \bD^{-1}\bH^{\top}\left(\bH \bD^{-1}\bH^{\top} \right)^{-1}\bH \bD^{-1}$ and $\bdelta = \bD^{-1} \bH^{\top}\left(\bH \bD^{-1} \bH^{\top} \right)^{-1} \bgamma$. Thus, similarly, we obtain
\begin{eqnarray*}
\mathcal{B}\left( \hbRE \right)&=&\mathbb{E}\left(\bvt_2\right) \\
&=&\mathbb{E}\left\{ \underset{n\rightarrow \infty }{\lim }\sqrt{n}\left( \hbRE- \bbeta \right) \right\} \\
&=&\mathbb{E}\left\{ \underset{n\rightarrow \infty }{\lim }\sqrt{n}\left( \bLd \hbRMLE- \bbeta \right) \right\} \\
&=&\mathbb{E}\left\{ \underset{n\rightarrow \infty }{\lim }\sqrt{n}\left(\left[ \hbUR-\bbeta \right] -\bLd \bD^{-1}\bH^{\top}\left(\bH \bD^{-1} \bH^{\top} \right)^{-1} \left( \bH \hbMLE - \bh \right) \right) \right\} \\
&=&\bzeta-\bLd \bdelta.
\end{eqnarray*} 

Now, we need to compute $\bGamma \left( \hbRE \right)$ which is as follows: 
\begin{eqnarray}\label{cov_RE}
\bGamma \left( \hbRE \right) &=&\mathbb{E}\left\{
\underset{n\rightarrow \infty }{\lim }{n}\left( \hbRE -\bbeta \right)\left(\hbRE -\bbeta \right) ^{\top}\right\} \no \\
&=&\mathbb{E}\left( \bvt_{2}\bvt_{2}^{\top}\right) \no \\
&=& \boldsymbol{Cov}\left( \bvt_{2}\bvt_{2}^{\top}\right) +\mathbb{E}\left( \bvt_{2}\right) \mathbb{E}\left( \bvt_{2}^{\top}\right) \no \\
&=& \bLd \left( \bD^{-1}-\bA \right) \bLd^\top + \left( \bzeta - \bLd\bdelta\right)\left( \bzeta - \bLd\bdelta\right)^{\top}\no \\
&=& \bLd \left( \bD^{-1}-\bA \right) \bLd^\top +\left[\left(1-d\right)^{2} \left(\bD+\bI \right)^{-2} \bbeta +\bL_{d} \bdelta \right]\left[\left(1-d\right)^{2} \left(\bD+\bI\right)^{-2} \bbeta +\bL_{d} \bdelta \right]^\top  \no\\
&=& \bLd \left( \bD^{-1}-\bA \right) \bLd^\top + \left(1-d\right)^{2} \left(\bD+\bI\right)^{-2} \bbeta \bdelta ^\top \bL_{d}^\top + \left(1-d\right)^{2} \bL_{d} \bdelta \bbeta ^\top \left(\bD+\bI\right)^{-2}  \\ 
&& + \bL_{d} \bdelta \bdelta ^\top \bL_{d}^\top +\left(1-d\right)^{4} \left(\bD+\bI\right)^{-2} \bbeta \bbeta ^\top \left(\bD+\bI\right)^{-2} \no 
\end{eqnarray}
Thus, $\bvt_{2}  \sim\mathcal{N} \left(\bzeta-\bLd \bdelta, \bLd \left( \bD^{-1}-A \right) \bLd^\top  \right)$. \\
\\
%%%%%%%%%%%%%%%%%%%%%%%%%%%%%%%%%%%%%%%%%%%%%%%%%%%%%%%
Moreover, to obtain the asymptotic distribution of $\bvt_3$, we start with the following:
\begin{eqnarray*}
\mathbb{E}\left(\bvt_3\right)&=&\mathbb{E}\left\{\underset{%
n\rightarrow \infty }{\lim }\sqrt{n}\left(\hbUR-\hbRE \right)\right\} \\ 
&=&\mathbb{E}\left\{ \underset{n\rightarrow \infty }{\lim }\sqrt{n}\bLd \left( \bD^{-1}\bH^{\top}\left(\bH \bD^{-1} \bH^{\top} \right)^{-1}\left( \bH \hbMLE - \bh \right)\right)\right\} \\
&=& \bLd \bdelta
\end{eqnarray*}
and to compute the covariance matrix of $\vartheta_3$ we compute as follows:
\begin{eqnarray*}
Cov\left(\bvt_3,\bvt_3^\top \right)&=&Cov\left( \hbUR-\hbRE \right) \\
&=&Cov\left( \bLd\left( \hbMLE-\hbRMLE \right)\right)\\
&=& \bLd \bA \bLd^\top
\end{eqnarray*}
Thus, $\bvt_3\sim{\mathcal{N}\left(\bLd \bdelta,\bLd \bA \bLd^\top \right)}$. \\
%%%%%%%%%%%%%%%%%%%%%%%%%%%%%
\\
Now, we also need to compute $Cov\left(\bvt_1,\bvt_3\right)$ and $Cov\left(\bvt_2,\bvt_3\right)$. We start with computing $Cov\left(\bvt_1,\bvt_3\right)$ as follows:
\begin{eqnarray*}
Cov\left(\bvt_1,\bvt_3\right)&=&Cov\left(\hbUR, \hbUR- \hbRE\right) \\
&=&Cov\left(\hbUR, \hbUR\right) -Cov\left(\hbUR,\hbRE\right)  \\
&=& \bLd \bD^{-1} \bLd^\top - Cov\left(\bLd \hbMLE,\bLd \hbRMLE \right) \\
&=& \bLd \bD^{-1} \bLd^\top - \bLd Cov\left(\hbMLE,\hbRMLE \right) \bLd^\top \\
&=& \bLd \bD^{-1} \bLd^\top - \bLd \left( \bD^{-1} - \bA \right)\bLd^\top \\
&=& \bLd \bA \bLd ^\top
\end{eqnarray*}

Finally, $Cov\left(\bvt_2,\bvt_3\right)$ is obtained as follows:
$Cov\left(\bvt_2,\bvt_3\right)$ is obtained as follows:
\begin{eqnarray*}
Cov\left(\bvt_2,\bvt_3\right)&=&Cov\left(\hbRE, \hbUR- \hbRE\right) \\
&=&Cov\left(\hbRE, \hbUR \right) - Cov\left(\hbRE, \hbRE \right) \\
&=& \bLd \left( \bD^{-1} - \bA \right) \bLd^\top - \bLd \left( \bD^{-1} - \bA \right)\bLd^\top\\
&=& \b0_{p+1}
\end{eqnarray*}
The proof is finished.
\end{proof}

%%%%%%%%%%%%%%%%%%%%%%%%%%%%%%%%%%
%%%%%%%%%%%%%%%%%%%%%%%%%%%%%%%%%%

\begin{proof}[Proof of Theorem \ref{bias}]
The asymptotic biases of UR ans RE are already obtained in the proof of Proposition \ref{prop_vector_dist}. 
Now, we continue computing the asymptotic bias of PT as follows:
%%%%%%%%%  bias_PT %%%%%%%%%%%%%%%%
\begin{eqnarray*}
\mathcal{B}\left( \hbPT \right)&=&\mathbb{E}\left\{ \underset{n\rightarrow \infty }{\lim }\sqrt{n}\left(\hbPT-\bbeta \right) \right\} \\
&=&\mathbb{E}\left\{ \underset{n\rightarrow \infty }{\lim }\sqrt{n}\left( \hbUR -\left(\hbUR -\hbRE \right)\textrm{I}\left[L_{n}<  \chi _{q,\alpha
}^{2}  \right]-\bbeta \right) \right\}\\
&=&\mathbb{E}\left\{ \underset{n\rightarrow \infty }{\lim }\sqrt{n}\left( \hbUR -\bbeta -\left(\hbUR -\hbRE \right)\textrm{I}\left[L_{n}<  \chi _{q,\alpha
}^{2}  \right] \right) \right\}\\
&=&\bzeta-\bLd\bdelta \mathbb{H}_{q+2}\left( \chi_{q,\alpha}^{2};\Delta^2 \right)\cr
\end{eqnarray*}
%%%%%%%%  bias_S   %%%%%%%%%%%%%%%%%
\begin{eqnarray*}
\mathcal{B}\left( \hbS \right)&=&\mathbb{E}\left\{ \underset{n\rightarrow \infty }{\lim }\sqrt{n}\left( \hbS -\bbeta \right) \right\} \\
&=&\mathbb{E}\left\{ \underset{n\rightarrow \infty }{\lim }\sqrt{n}\left( \hbUR -\left(\hbUR-\hbRE \right)cL_{n}^{-1} - \bbeta\right) \right\}\\
&=&\mathbb{E}\left\{ \underset{n\rightarrow \infty }{\lim }\sqrt{n}\left( \hbUR -\bbeta \right) \right\} -\mathbb{E}\left\{ \underset{n\rightarrow \infty }{\lim }\sqrt{n}\left(\hbUR-\hbRE \right)cL_{n}^{-1} \right\}\\
&=&\bzeta-c\bLd\bdelta \mathbb{E}\left( \chi_{q+2}^{-2}(\Delta^2)\right)
\end{eqnarray*}
%%%%%%%%%%%%%%%%%%%%%%%%%%%%%%%%%%%%%%%%%%%%%%%%%%%%%%%%%%%%%%%%
\begin{eqnarray*}
\mathcal{B}\left( \hbPS \right)&=&\mathbb{E}\left\{ \underset{n\rightarrow \infty }{\lim }\sqrt{n}\left(  \hbS -\bbeta \right) \right\} \\
&=&\mathbb{E}\left\{ \underset{n\rightarrow \infty }{\lim }\sqrt{n}\left( \hbS -\left(\hbUR-\hbRE \right)\left(1-c L_{n}^{-1}\right)
\textrm{I}\left(L_{n}\leq c \right)-\bbeta\right) \right\}\\
&=&\mathbb{E}\left\{ \underset{n\rightarrow \infty }{\lim }\sqrt{n}\left(\hbS-\bbeta \right) \right\}-\mathbb{E}\left\{ \underset{n\rightarrow \infty }{\lim }\sqrt{n}\left( \hbUR-\hbRE \right)\textrm{I}(L_{n}\leq c ) \right\}\\
&& + \mathbb{E}\left\{ \underset{n\rightarrow \infty }{\lim }\sqrt{n}\left( \hbUR-\hbRE \right)cL_{n}^{-1}\textrm{I}(L_{n}<c)\right\}\\
&=&\bzeta-\bLd \bdelta\mathbb{E}\left( \chi _{q+2}^{-2}(\Delta^2) \right)-\bLd \bdelta \mathbb{H}_{q+2}\left( \chi _{q,\alpha}^{2};\Delta^2 \right) \\
&&+c\bLd \bdelta \mathbb{E}\left( \chi _{q+2}^{-2}(\Delta^2)\textrm{I}\left(\chi _{q+2}^{2}(\Delta^2) < c\right) \right)
\end{eqnarray*}
\end{proof}
%%%%%%%%%%%%%%%%%%%%%%%%%%%%%%%%%%%%%%%%%%%%%%%%%%%%%%%%%%%%%%
%%%%%%%%%%%%%%%%%%%%%%%%%%%%%%%%%%%%%%%%%%%%%%%%%%%%%%%%%%%%%%
\begin{proof}[Proof of Theorem ~\ref{risks}]
The asymptotic covariance of $\hbUR$ and $\hbRE $ are already obtained in the proof of Theorem \ref{prop_vector_dist} respectively in Eqns. ~\ref{cov_UR} and ~\ref{cov_RE}.
Now, we continue with the covariance of $\hbPT$ as follows:
\begin{eqnarray*}
    \bGamma \left(\hbPT \right) &=& \mathbb{E}\left\{ \underset{n\rightarrow \infty }{\lim }\sqrt{n}\left(  \hbPT -\bbeta \right)\left(  \hbPT -\bbeta \right)^{\top} \right\} \\
 &=& \mathbb{E}\left\{ \underset{n\rightarrow \infty }{\lim }\sqrt{n}\left[ \left(\hbUR -\bbeta \right) -\left(\hbUR-\hbRE\right)\tI \left(L_{n} <\chi _{q,\alpha}^{2} \right)\right]\right. \\
 && \left. \times \left[ \left(\hbUR -\bbeta \right) -\left(\hbUR-\hbRE\right)\tI \left(L_{n} <\chi _{q,\alpha}^{2} \right)\right]^{\top} \right\} \\
 &=& \E \left\{ \bvt_1 \bvt_1^{\top}-2\bvt_3 \bvt_1^{\top} \tI \left(L_{n} <\chi _{q,\alpha}^{2} \right)+\bvt_3 \bvt_3^{\top} \tI \left(L_{n} <\chi _{q,\alpha}^{2} \right)\right\}.
\end{eqnarray*}
We already have  $\E\left\{ \bvt_1 \bvt_1^{\top}\right\}=\bLd \bD^{-1} \bLd^\top + \bzeta\bzeta^\top$. We need the followings:
\begin{eqnarray*}
\E \left\{\bvt_3 \bvt_3^{\top} \tI \left(L_{n} <\chi _{q,\alpha}^{2} \right) \right\}= \bLd \bA \bLd^\top  \H_{q+2}\left( \chi_{q,\alpha}^{2};\Delta^2 \right)+ \bLd \bdelta \bdelta^{\top} \bLd^\top \H_{q+4}\left( \chi_{q,\alpha}^{2};\Delta^2 \right)
\end{eqnarray*}
and
\begin{eqnarray*}
\E \left\{\bvt_3 \bvt_1^{\top} \tI \left(L_{n} <\chi _{q,\alpha}^{2} \right) \right\}&=& \E \left\{\bvt_3 \E\left[ \bvt_1^{\top} \tI \left(L_{n} <\chi _{q,\alpha}^{2} \right)|\bvt_3\right] \right\}\\
&=&\E \left\{\bvt_3 \left[\bzeta+\bvt_3-\bL_d\bdelta \right]^{\top}\tI \left(L_{n} <\chi _{q,\alpha}^{2} \right) \right\}\\
&=&\E \left\{\left(\bvt_3 \bzeta ^{\top}+\bvt_3\bvt_3^{\top}-\bvt_3\bdelta^{\top} \bL_d^\top \right) \tI \left(L_{n} <\chi _{q,\alpha}^{2} \right) \right\}\\
&=& \left(\bL_d\bdelta\bzeta^{\top} +\bL_d \bA \bL_d^\top -\bL_d\bdelta\bdelta^{\top}\bL_d^\top \right)\H_{q+2}\left( \chi_{q,\alpha}^{2};\Delta^2 \right)+\bL_d\bdelta\bdelta^{\top}\bL_d^\top \H_{q+4}\left( \chi_{q,\alpha}^{2};\Delta^2 \right)
\end{eqnarray*}
Thus, we obtain
\begin{eqnarray}\label{cov_PT}
\bGamma \left(\hbPT \right) &=& \bLd \bD^{-1} \bLd^\top + \bzeta\bzeta^\top - \bL_d \bA \bL_d^\top\H_{q+2}\left( \chi_{q,\alpha}^{2};\Delta^2 \right)-2\bL_d\bdelta\bzeta^{\top}\H_{q+2}\left( \chi_{q,\alpha}^{2};\Delta^2 \right)\no\\
&&+\bL_d\bdelta\bdelta^{\top}\bL_d^\top \left[2\H_{q+2}\left( \chi_{q,\alpha}^{2};\Delta^2 \right)-\H_{q+4}\left( \chi_{q+4}^{2};\Delta^2 \right)\right].
\end{eqnarray}
In a similar manner, the asymptotic covariance of $\hbS$ can be obtained as follows:
\begin{eqnarray*}
    \bGamma \left(\hbS \right) &=& \mathbb{E}\left\{ \underset{n\rightarrow \infty }{\lim }\sqrt{n}\left(  \hbS -\bbeta \right)\left(  \hbS -\bbeta \right)^{\top} \right\} \\
 &=& \mathbb{E}\left\{ \underset{n\rightarrow \infty }{\lim }\sqrt{n}\left[ \left(\hbUR -\bbeta \right) -c\left(\hbUR-\hbRE\right) L_{n}^{-1}\right]\right. \\
 && \left. \times \left[ \left(\hbUR -\bbeta \right) -c\left(\hbUR-\hbRE\right) L_{n}^{-1}\right]^{\top} \right\} \\
 &=& \E \left\{ \bvt_1 \bvt_1^{\top}-2c\bvt_3 \bvt_1^{\top} L_{n}^{-1}+c^2 \bvt_3 \bvt_3^{\top} L_{n}^{-2}\right\}.
\end{eqnarray*}
Thus, we need the following identities:
\begin{eqnarray*}
\E \left\{\bvt_3 \bvt_3^{\top} L_{n}^{-2} \right\} = \bLd \bA \bLd ^\top \E\left( \chi_{q+2}^{-4}\left(\Delta^2\right) \right)+ \bLd \bdelta \bdelta^{\top} \bLd^\top \E\left( \chi_{q+4}^{-4}\left(\Delta^2\right) \right)
\end{eqnarray*}
and
\begin{eqnarray*}
\E \left\{\bvt_3 \bvt_1^{\top} L_{n}^{-1}\right\}&=& \E \left\{\bvt_3 \E\left[ \bvt_1^{\top} L_{n}^{-1} | \bvt_3\right] \right\}\\
&=&\E \left\{\bvt_3 \left[\bzeta+\bvt_3-\bL_d\bdelta \right]^{\top}L_{n}^{-1} \right\}\\
&=&\E \left\{\bvt_3 \bzeta ^{\top}L_{n}^{-1}+\bvt_3\bvt_3^{\top}L_{n}^{-1}-\bvt_3\bdelta^{\top} \bL_d^{\top} L_{n}^{-1} \right\}\\
&=& \E\left( \chi_{q+2}^{-2}\left(\Delta^2\right) \right)\left(\bL_d\bdelta\bzeta^{\top} +\bL_d \bA \bL_d^{\top} -\bL_d\bdelta\bdelta^{\top}\bL_d ^{\top}\right)+\E\left( \chi_{q+4}^{-2}\left(\Delta^2\right) \right)\bL_d\bdelta\bdelta^{\top}\bL_d^{\top}.
\end{eqnarray*}
Therefore, we obtain
\begin{eqnarray}\label{cov_S}
\bGamma \left(\hbS \right) &=& \bLd \bD^{-1} \bLd^\top + \bzeta\bzeta^\top-\bL_d \bA \bL_d ^\top \left( c^2\E\left\{ \chi_{q+2}^{-4}\left(\Delta^2\right) \right\}-2c\E\left\{\chi_{q+2}^{-2}\left(\Delta^2\right)\right\}\right)\no \\
&&+\bL_d\bdelta\bdelta^{\top}\bL_d ^{\top}\left(c^2\E\left\{ \chi_{q+4}^{-4}\left(\Delta^2\right) \right\}+2c\E\left\{\chi_{q+2}^{2}\left(\Delta^2\right) \right\}-2c\E\left\{\chi_{q+4}^{-2}\left(\Delta^2\right) \right\}  \right)\no\\
&&-2c\bL_d\bdelta\bzeta^{\top}\E\left\{ \chi_{q+2}^{2}\left(\Delta^2\right) \right\} 
\end{eqnarray}
Finally, we present the asymptotic covariance of $\hbPS$ as follows:
\begin{eqnarray*}
    \bGamma \left(\hbPS \right) &=& \mathbb{E}\left\{ \underset{n\rightarrow \infty }{\lim }\sqrt{n}\left(  \hbPS -\bbeta \right)\left(  \hbPS -\bbeta \right)^{\top} \right\} \\
 &=& \mathbb{E}\left\{ \underset{n\rightarrow \infty }{\lim }\sqrt{n}\left[ \left(\hbS -\bbeta \right) -(1-cL_{n}^{-1})\left(\hbUR-\hbRE\right)\tI \left(L_{n} < c \right)  \right]\right. \\
 && \left. \times \left[ \left(\hbS -\bbeta \right) -(1-cL_{n}^{-1})\left(\hbUR-\hbRE\right)\tI \left(L_{n} < c \right) \right]^{\top} \right\} \\
 &=& \bGamma \left(\hbS \right) + \E \left\{ \bvt_3 \bvt_3^{\top} (1-cL_{n}^{-1})^2 \tI \left(L_{n} < c \right)\right\}\\
 &&-2 \E \left\{  \bvt_3 \bvt_1^{\top} (1-cL_{n}^{-1})  \tI \left(L_{n} < c \right)\right\}\\
 &&+ 2c\E \left\{ c^2 \bvt_3 \bvt_3^{\top} L_{n}^{-1}(1-cL_{n}^{-1})\tI \left(L_{n} < c \right)\right\}\\
 &=& \bGamma \left(\hbS \right) + \E \left\{ \bvt_3 \bvt_3^{\top}\tI \left(L_{n} < c \right)\right\}-2 \E \left\{  \bvt_3 \bvt_1^{\top} (1-cL_{n}^{-1})  \tI \left(L_{n} < c \right)\right\}\\
 &&-2c^2\E \left\{  \bvt_3 \bvt_3^{\top} L_{n}^{-2} \tI \left(L_{n} < c \right)\right\}.
\end{eqnarray*}
Now we need the following identities:
\begin{eqnarray*}
\E \left\{\bvt_3 \bvt_3^{\top} \tI \left(L_{n} < c\right) \right\}= \bLd \bA \bLd^{\top}  \H_{q+2}\left( c;\Delta^2 \right)+ \bLd \bdelta \bdelta^{\top} \bLd^{\top} \H_{q+4}\left(c;\Delta^2 \right)
\end{eqnarray*}
and 
\begin{eqnarray*}
\E \left\{\bvt_3 \bvt_3^{\top} L_{n}^{-2} \tI \left(L_{n} < c \right) \right\} &=& \bLd \bA \bLd^{\top}  \E \left\{ \chi_{q+2}^{-4}\left(\Delta^2\right) \tI \left(\chi_{q+2}^{-4}\left(\Delta^2\right) < c\right)\right\}\\
&&+ \bLd \bdelta \bdelta^{\top} \bLd \E\left\{ \chi_{q+4}^{-4}\left(\Delta^2\right) \tI \left(\chi_{q+4}^{2}\left(\Delta^2\right) <c \right)\right\}.
\end{eqnarray*}
Moreover, we also have,
\begin{eqnarray*}
\E \left\{\bvt_3 \bvt_1^{\top} \left(1-cL_{n}^{-1}\right) \tI \left(L_{n} < c \right) \right\}&=& \E \left\{\bvt_3 \E\left[ \bvt_1^{\top} \left(1-cL_{n}^{-1}\right) \tI \left(L_{n} < c \right) | \bvt_3\right] \right\}\\
&=&\E \left\{\bvt_3 \left[\bzeta+\bvt_3-\bL_d\bdelta \right]^{\top}\left(1-cL_{n}^{-1}\right) \tI \left(L_{n} < c \right) \right\}\\
&=&\E \left\{\left(\bvt_3 \bzeta ^{\top}+\bvt_3\bvt_3^{\top}-\bvt_3\bdelta^{\top} \bL_d\right) \left(1-cL_{n}^{-1}\right) \tI \left(L_{n} < c \right) \right\}\\
&=& \bL_d\bdelta \bzeta^{\top} \E\left\{ \left(1-c\chi_{q+2}^{-2}\left(\Delta^2\right) \right)\tI \left(\chi_{q+2}^{-2}\left(\Delta^2\right)  < c \right)\right\}\\
&&+\bL_d \bA \bL_d^\top \E\left\{ \left(1-c\chi_{q+2}^{-2}\left(\Delta^2\right) \right)\tI \left(\chi_{q+2}^{-2}\left(\Delta^2\right)  < c \right)\right\}\\
&&+\bL_d\bdelta\bdelta^{\top}\bL_d\left[\E\left\{ \left(1-c\chi_{q+4}^{-2}\left(\Delta^2\right) \right)\tI \left(\chi_{q+4}^{-2}\left(\Delta^2\right)  < c \right)\right\}\right]\\
&&-\bL_d\bdelta\bdelta^{\top}\bL_d ^{\top} \left[ \E\left\{ \left(1-c\chi_{q+2}^{-2}\left(\Delta^2\right) \right)\tI \left(\chi_{q+2}^{-2}\left(\Delta^2\right)  < c \right)\right\} \right].
\end{eqnarray*}
Thus, we finally obtain
\begin{eqnarray}\label{cov_PS}
    \bGamma \left(\hbPS \right) &=& \bGamma \left(\hbS \right)+\bLd \bA \bLd ^{\top} \H_{q+2}\left( c;\Delta^2 \right)+ \bLd \bdelta \bdelta^{\top} \bLd^{\top}  \H_{q+4}\left(c;\Delta^2 \right)\no \\
    &&-c^2 \bLd \bA \bLd^{\top} \E\left\{ \left(\chi_{q+2}^{-4}\left(\Delta^2\right) \right)\tI \left(\chi_{q+2}^{-2}\left(\Delta^2\right)  < c \right)\right\} \no \\
    && -c^2 \bLd \bdelta \bdelta^{\top} \bLd^{\top} \E\left\{ \left(\chi_{q+4}^{-4}\left(\Delta^2\right) \right)\tI \left(\chi_{q+4}^{-2}\left(\Delta^2\right)  < c \right)\right\} \no \\
    &&-2\left(\bL_d\bdelta \bzeta^{\top} +\bL_d \bA \bL_d^{\top} -\bL_d\bdelta\bdelta^{\top}\bL_d ^{\top}\right)\E\left\{ \left(1-c\chi_{q+2}^{-2}\left(\Delta^2\right) \right)\tI \left(\chi_{q+2}^{-2}\left(\Delta^2\right)  < c \right)\right\} \no \\
    &&-2\bL_d\bdelta\bdelta^{\top}\bL_d ^{\top}\E\left\{ \left(1-c\chi_{q+4}^{-2}\left(\Delta^2\right) \right)\tI \left(\chi_{q+4}^{-2}\left(\Delta^2\right)  < c \right)\right\}
\end{eqnarray}

Now, using the asymptotic covariances and the definition of the risk of an estimator, we can easily obtain the risk functions of the estimators. Let $\bW = \bI$, using Eqn. (\ref{cov_UR}), making use of the spectral decomposition of $\bD$, there exists an orthogonal matrix $\mathrm{\bQ }$ such that $\bQ^\top \bD \bQ = \bLambda = {\rm diag}\left(\lambda _{1} ,\ldots ,\lambda _{p+1} \right)$, where $\lambda _{1} \ge \ldots \ge \lambda _{p+1} >0$ are the eigenvalues of $\bD$, we have the following
\begin{eqnarray}\label{risk_UR}
\mathcal{R}\left( \hbUR \right)&=&
{\rm tr}\left[\bGamma\left(\hbUR \right)\right]\no\\
&=&{\rm tr}\left[ \bLd \bD^{-1} \bLd^\top + \bzeta \bzeta^\top \right]\no\\
&=&{\rm tr}\left[\bLd \bD^{-1} \bLd^\top + \left(1-d\right)^{4} \left(\bD+\bI\right)^{-2} \bbeta \bbeta ^\top \left(\bD+\bI\right)^{-2}  \right]\no\\
&=& \sum_{j=1}^{p+1} \left\{\frac{\left(\lambda _{j} +d\right)^{2} \left(\lambda _{j} +d-2\right)^{2} }{\lambda _{j} \left(\lambda _{j} +1\right)^{4} }+ \frac{\left(1-d\right)^{4} \theta _{j}^{2} }{\left(\lambda _{j} +1\right)^{4} }\right\}
\end{eqnarray}
where $\btheta = \bQ^\top \bbeta =\left(\theta _{1}, \ldots , \theta _{p+1} \right)^{\top}$ and $${\rm tr}\left[\bLd \bD^{-1} \bLd^\top \right] = \sum_{j=1}^{p+1} \left\{\frac{\left(\lambda _{j} +d\right)^{2} \left(\lambda _{j} +d-2\right)^{2} }{\lambda _{j} \left(\lambda _{j} +1\right)^{4} }\right\}$$ and $${\rm tr}\left[\bzeta \bzeta ^\top \right] = \sum_{j=1}^{p+1} \left\{\frac{\left(1-d\right)^{4} \theta _{j}^{2} }{\left(\lambda _{j} +1\right)^{4}}\right\}.$$

The risk of the restricted estimator RE is obtained using Eqn. (\ref{cov_RE}) as follows:

\begin{eqnarray}\label{risk_RE}
\mathcal{R}\left( \hbRE \right) 
&=& {\rm tr}\left[ \bGamma\left(\hbRE \right)\right] \no\\
&=& {\rm tr}\left[ \bLd \left( \bD^{-1}-\bA \right) \bLd ^\top \right] + \left[\left(1-d\right)^{2} \left(\bD+\bI \right)^{-2} \bbeta +\bLd \bdelta \right]^\top \left[\left(1-d\right)^{2} \left(\bD+\bI \right)^{-2} \bbeta +\bL_{d} \bdelta \right] \no \\
&=& {\rm tr}\left[ \bLd \left( \bD^{-1}-\bA \right)\bLd^\top \right]+ \left(1-d\right)^{4} \bbeta^\top \left(\bD+\bI \right)^{-4} \bbeta +  \left(1-d\right)^{2} \bbeta ^\top \left(\bD+\bI \right)^{-2} \bLd \bdelta \no \\
&& + \left(1-d\right)^{2}  \bdelta^\top \bLd^\top \left(\bD+\bI \right)^{-2} \bbeta +  \bdelta ^\top \bLd^\top \bLd\bdelta\no \\
&=& \sum_{j=1}^{p+1} \left\{\frac{\left(\lambda_j+2-d\right)^2\left(\lambda_j+d\right)^2}{\left(\lambda_j+1\right)^4} \left[\frac{1-\lambda_j a_{jj}}{\lambda_j}+\delta_j^2 \right] + \frac{\left(1-d\right)^2}{\left(\lambda_j+1\right)^4}\left[ \beta_j^2 + 2\beta_j^2\delta_j \left(\lambda_j+2-d\right) \left(\lambda_j+d\right)\right] \right\}\no\\
\end{eqnarray}
where $a_{jj}$ is the jth diagonal element of the matrix $\bA$.

The asymptotic risk of the preliminary test estimator PT is computed using Eqn. (\ref{cov_PT}) as follows:
\begin{eqnarray}\label{risk_PT}
\mathcal{R}\left( \hbPT \right)&=&
{\rm tr}\left[\bGamma\left(\hbPT \right)\right]\no\\
&=&\mathcal{R}\left( \hbUR \right)-2{\rm tr}\left[\bL_d \bA \bL_d^\top\right]\H_{q+2}\left( \chi_{q,\alpha}^{2};\Delta^2 \right)-2{\rm tr}\left[\bL_d\bdelta\bzeta^{\top}\right]\H_{q+2}\left( \chi_{q,\alpha}^{2};\Delta^2 \right)\no\\
&&+{\rm tr}\left[\bL_d\bdelta\bdelta^{\top}\bL_d^\top\right] \left[2\H_{q+2}\left( \chi_{q,\alpha}^{2};\Delta^2 \right)-\H_{q+4}\left( \chi_{q+4}^{2};\Delta^2 \right)\right]
\end{eqnarray}
where $${\rm tr}\left[\bLd \bdelta \bdelta^\top \bLd^\top \right] = \sum_{j=1}^{p+1} \left\{\frac{\left(\lambda _{j} +d\right)^{2} \left(\lambda _{j} +d-2\right)^{2} }{ \left(\lambda _{j} +1\right)^{4} }\delta_j^2\right\}$$  and 
$${\rm tr}\left[\bLd \bdelta \bzeta^\top \right] = \sum_{j=1}^{p+1} \left\{(1-d)^2\frac{\left(\lambda _{j} +d\right) \left(\lambda _{j} +d-2\right) }{ \left(\lambda _{j} +1\right)^{4} }\delta_j \beta_j\right\}. $$
The asymptotic risk of the shrinkage estimator S is computed using Eqn. (\ref{cov_S}) as follows:
\begin{eqnarray}\label{risk_S}
\mathcal{R}\left( \hbS \right)&=&
{\rm tr}\left[\bGamma\left(\hbS \right)\right]\no\\
&=&{\rm tr}\left[\bLd \bD^{-1} \bLd^\top -2c\left(\bL_d\bdelta\bzeta^{\top} +\bL_d \bA \bL_d^\top-\bL_d\bdelta\bdelta^{\top}\bL_d^\top \right)\E\left( \chi_{q+2}^{-2}(\Delta^2) \right)\right. \no \\
&&\left.-2c\bL_d\bdelta\bdelta^{\top}\bL_d^\top \E\left( \chi_{q+4}^{-2}(\Delta^2) \right) + c^2\bL_d \bA \bL_d ^\top\E\left( \chi_{q+2}^{-4}(\Delta^2) \right) \no \right.\\
&&\left. +c^2 \bL_d\bdelta\bdelta^{\top}\bL_d ^\top \E\left( \chi_{q+4}^{-4}(\Delta^2) \right) \no \right]\\
&=& \mathcal{R}\left( \hbUR \right)+ {\rm tr}\left[\bLd \bA \bLd^\top\right] \left( c^2\E\left\{ \chi_{q+2}^{-4}\left(\Delta^2\right) \right\}-2c\E\left\{\chi_{q+2}^{-2}\left(\Delta^2\right)\right\}\right) \no\\
&&+{\rm tr}\left[\bLd \bdelta \bdelta^\top \bLd^\top \right] \left(c^2\E\left\{ \chi_{q+4}^{-4}\left(\Delta^2\right) \right\}+2c\E\left\{\chi_{q+2}^{2}\left(\Delta^2\right) \right\}-2c\E\left\{\chi_{q+4}^{-2}\left(\Delta^2\right) \right\}  \right) \no \\
&&-2c{\rm tr}\left[ \bLd \bdelta \bzeta^\top\right]\E\left\{ \chi_{q+2}^{2}\left(\Delta^2\right) \right\}.
\end{eqnarray}
Finally, the asymptotic risk of the estimator PS is obtained using Eqn. (\ref{cov_PS}) as follows:
\begin{eqnarray}\label{risk_PS}
\mathcal{R}\left( \hbPS \right)&=&
{\rm tr}\left[\bGamma\left(\hbPS \right)\right]\no\\
&=&{\rm tr}\left[\bGamma \left(\hbS \right)+\bLd \bA \bLd^\top  \H_{q+2}\left( c;\Delta^2 \right)+ \bLd \bdelta \bdelta^{\top} \bLd^\top \H_{q+4}\left(c;\Delta^2 \right)\no \right.\\
    && \left.-c^2 \bLd \bA \bLd^\top \E\left\{ \left(\chi_{q+2}^{-4}\left(\Delta^2\right) \right)\tI \left(\chi_{q+2}^{-2}\left(\Delta^2\right)  < c \right)\right\} \no \right. \\
    && \left.-c^2 \bLd \bdelta \bdelta^{\top} \bLd^\top \E\left\{ \left(\chi_{q+4}^{-4}\left(\Delta^2\right) \right)\tI \left(\chi_{q+4}^{-2}\left(\Delta^2\right)  < c \right)\right\} \no \right.\\
    &&\left.-2\left(\bL_d\bdelta \bzeta^{\top} +\bL_d \bA \bL_d^\top -\bL_d\bdelta\bdelta^{\top}\bL_d^\top \right)\E\left\{ \left(1-c\chi_{q+2}^{-2}\left(\Delta^2\right) \right)\tI \left(\chi_{q+2}^{-2}\left(\Delta^2\right)  < c \right)\right\} \no\right. \\
    &&\left.-2\bL_d\bdelta\bdelta^{\top}\bL_d^\top \E\left\{ \left(1-c\chi_{q+4}^{-2}\left(\Delta^2\right) \right)\tI \left(\chi_{q+4}^{-2}\left(\Delta^2\right)  < c \right)\right\} \right] \no\\
    &=& \mathcal{R}\left( \hbS\right) \no\\
  &&+ {\rm tr}\left[ \bLd \bA \bLd^\top \right]\left[ \H_{q+2}\left( c;\Delta^2 \right)-c^2 \E\left\{ \left(\chi_{q+2}^{-4}\left(\Delta^2\right) \right)\tI \left(\chi_{q+2}^{-2}\left(\Delta^2\right)  < c \right)\right\} \no \right.\\
  &&\left.-2\E\left\{ \left(1-c\chi_{q+2}^{-2}(\Delta^2) \right)\tI \left(\chi_{q+2}^{-2}\left(\Delta^2\right)  < c \right)\right\}\right] \no\\
  &&+ {\rm tr}\left[ \bLd \bdelta \bdelta^\top \bLd^\top \right] \left[\H_{q+2}\left( c;\Delta^2 \right)-c^2\E\left\{ \left(\chi_{q+4}^{-4}(\Delta^2) \right)\tI \left(\chi_{q+2}^{-2}(\Delta^2)  < c \right)\right\}  \no \right.\\
  &&\left. +2c\E\left\{ \left(1-c\chi_{q+2}^{-2}(\Delta^2) \right)\tI \left(\chi_{q+2}^{-2}(\Delta^2)  < c \right)\right\} -2\E\left\{ \left(1-c\chi_{q+4}^{-2}(\Delta^2) \right)\tI \left(\chi_{q+4}^{-2}(\Delta^2)  < c \right)\right\}\right] \no\\
 && -2{\rm tr}\left[ \bLd \bdelta \bzeta^\top \right]\E\left\{ \left(1-c\chi_{q+4}^{-2}(\Delta^2) \right)\tI \left(\chi_{q+4}^{-2}(\Delta^2)  < c \right)\right\}
\end{eqnarray}
\end{proof}
%%%%%%%%%%%%%%%%%%%%%%%%%%%%%%%%%%%%%%

\end{document}